\documentclass[11pt]{amsart}
\usepackage[utf8]{inputenc}
\usepackage{amsmath,amssymb,amsthm,mathdots}
\usepackage{hyperref}
\usepackage[a4paper,hmargin=2.5cm,vmargin=2.5cm]{geometry}
\usepackage[dvips,dvipdfm,pdftex]{graphicx}

\usepackage{tikz}
\usepackage{booktabs}
\usepackage{mathrsfs}
\usepackage{stmaryrd}
\usepackage{float}
\usepackage{todonotes}

\newcommand{\cross}{\includegraphics[scale=0.5]{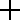}}
\newcommand{\elbows}{\includegraphics[scale=0.5]{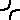}}
\newcommand{\elbow}{\includegraphics[scale=0.5]{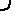}}

\newlength\cellsize 

\savebox2{%
	\begin{picture}(12,12)
	\put(0,0){\line(1,0){12}}
	\put(0,0){\line(0,1){12}}
	\put(12,0){\line(0,1){12}}
	\put(0,12){\line(1,0){12}}
	\end{picture}}

\newcommand\cellify[1]{\def\thearg{#1}\def\nothing{}%
	\ifx\thearg\nothing\vrule width0pt height\cellsize depth0pt%
	\else\hbox to 0pt{\usebox2\hss}\fi%
	\vbox to 12\unitlength{\vss\hbox to 12\unitlength{\hss$#1$\hss}\vss}}

\newcommand\tableau[1]{\vtop{\let\\=\cr
		\setlength\baselineskip{-12000pt}
		\setlength\lineskiplimit{12000pt}
		\setlength\lineskip{0pt}
		\halign{&\cellify{##}\cr#1\crcr}}}

\newcommand{\CC}{{\mathbb{C}}}

\newcommand{\ZZ}{{\mathbb{Z}}}

\newcommand{\cO}{{\mathcal{O}}}
\newcommand{\cP}{{\mathcal{P}}}
\newcommand{\cQ}{{\mathcal{Q}}}
\newcommand{\cS}{{\mathcal{S}}}
\newcommand{\cT}{{\mathcal{T}}}

\DeclareMathOperator{\GL}{GL}

\DeclareMathOperator{\word}{word}

\theoremstyle{plain}
\newtheorem{theorem}{Theorem}[section]

\newtheorem{proposition}[theorem]{Proposition}

\newtheorem{corollary}[theorem]{Corollary}

\theoremstyle{definition}
\newtheorem{definition}[theorem]{Definition}
\newtheorem{example}[theorem]{Example}

\theoremstyle{remark}
\newtheorem{remark}[theorem]{Remark}

\DeclareMathOperator{\del}{del}
\DeclareMathOperator{\link}{link}

\DeclareMathOperator{\PD}{PD}
\DeclareMathOperator{\QPD}{QPD}
\DeclareMathOperator{\dst}{dst}

\newcommand{\xx}{\mathbf x}
\newcommand{\Sc}{\mathbf{S}}
\newcommand{\Sf}{\mathfrak S}
\newcommand{\Ff}{\mathfrak F}
\newcommand{\Gf}{\mathfrak G}
\newcommand{\Gc}{\mathcal G}

\title{Slide polynomials and subword complexes}
\author{Evgeny Smirnov}
\email{esmirnov@hse.ru}
\address{HSE University, Russian Federation, ul. Usacheva 6, 119048 Moscow, Russia}
\address{Independent University of Moscow, Bolshoi Vlassievskii per. 11, 119002 Moscow, Russia}
\author{Anna Tutubalina}
\email{anna.tutubalina@gmail.com}
\address{HSE University, Russian Federation, ul. Usacheva 6, 119048 Moscow, Russia}
\date{\today}

\begin{document}

\begin{abstract}
Subword complexes were defined by A.\,Knutson and E.\,Miller in 2004 for describing Gröbner degenerations of matrix Schubert varieties. Subword complexes of a certain type are called pipe dream complexes. The facets of such a complex are indexed by pipe dreams, or, equivalently, by the monomials in the corresponding Schubert polynomial. In 2017, 
S.\,Assaf and D.\,Searles defined a basis of slide polynomials, generalizing Stanley symmetric functions, and described a combinatorial rule for expanding Schubert polynomials in this basis. We describe a decomposition of subword complexes into strata called slide complexes.The slide complexes appearing in such a way are shown to be  homeomorphic to balls or spheres.  For pipe dream complexes, such strata correspond to slide polynomials. 
\end{abstract}

\maketitle

\section{Introduction}\label{sec:intro}

\subsection{Schubert polynomials and pipe dreams}
Schubert polynomials $\Sf_w\in\ZZ[x_1,x_2,\dots]$ were defined by I.\,N.\,Bernstein, I.\,M.\,Gelfand and S.\,I.\,Gelfand \cite{BernsteinGelfandGelfand73} and by A.\,Lascoux and M.-P.\,Schützenberger \cite{LascouxSchutzenberger82}. They can be viewed as ``especially nice'' polynomial representatives of classes of Schubert varieties $[X_w]\in H^*(G/B)$, where $G=\GL_n(\CC)$ is a general linear group, $B$ is a Borel subgroup in $G$, and $G/B$ is a full flag variety. It is well known that their coefficients are nonnegative, and there exists a manifestly positive combinatorial rule for computing these coefficients. 

 One can also be interested in the $K$-theory $K_0(G/B)$. Instead of Schubert classes $[X_w]\in H^*(G/B)$, one would consider the classes of their structure sheaves $[\cO_w]\in K_0(G/B)$. These classes also have a nice presentation, known as Grothendieck polynomials $\Gf^{(\beta)}_w\in\ZZ[\beta,x_1,x_2,\dots]$, depending on an additional parameter $\beta$. They also have integer nonnegative coefficients, but, as opposed to Schubert polynomials, they are not homogeneous in the usual sense; however, they become homogeneous if we set $\deg\beta=-1$. They can be viewed as ``deformations'' of the Schubert polynomials $\Sf_w$: evaluating $\Gf^{(\beta)}_w$ at $\beta=0$, we recover the corresponding Schubert polynomial $\Sf_w=\Gf^{(0)}_w$.

Schubert and Grothendieck polynomials can be described combinatorially by means of diagrams called \emph{pipe dreams}, or \emph{rc-graphs}. These diagrams are configurations of pseudolines associated to a permutation; to each such diagram one can assign a monomial. A pipe dream is said to be \emph{reduced} if each pair of pseudolines intersects at most once. The Schubert (resp. Grothendieck) polynomial for a permutation $w$ is obtained as the sum of the corresponding monomials  for reduced (resp. not necessarily reduced) pipe dreams associated to $w$. This theorem, due to S.\,Billey and N.\,Bergeron \cite{BilleyBergeron93} and to S.\,Fomin and An.\,Kirillov \cite{FominKirillov96}, is an analogue of Littlewood's presentation of Schur polynomials as sums over Young tableaux. In particular, this implies positivity of the coefficients of Schubert and Grothendieck polynomials. A brief reminder on pipe dreams is given in \S~\ref{ssec:pipedreams}.

In \cite{KnutsonMiller05}, A.\,Knutson and E.\,Miller proposed a geometric interpretation of pipe dreams for a permutation $w$: they correspond to the irreducible components of a ``deep'' Gröbner degeneration of the corresponding matrix Schubert variety $\overline{X_w}$ to a union of affine subspaces. A combinatorial structure of this union of subspaces is encoded by a certain simplicial complex, known as the \emph{pipe dream complex} for $w$. From this, one can deduce that the multidegree of $\overline{X_w}$ with respect to the maximal torus $T\subset B\subset G$ equals the Schubert polynomial $\Sf_w$. 

In the subsequent paper \cite{KnutsonMiller04} the same authors put the notion of a pipe dream complex into a more general context, defining \emph{subword complexes} for an arbitrary Coxeter system, and proved that such complexes are shellable and, moreover,  homeomorphic to balls or, in certain ``rare'' cases, to spheres. This implies many interesting results about the geometry of the corresponding Schubert varieties, both matrix and usual ones, including new proofs for normality and Cohen--Macaulayness of Schubert varieties in a full flag variety.

\subsection{Slide and glide polynomials} Recently, S.\,Assaf and D.\,Searles \cite{AssafSearles17} defined \emph{slide polynomials} $\Ff_Q$. This is another family of polynomials with properties similar to Schubert polynomials: in particular, they form a basis in the ring of polynomials in countably many variables and enjoy a \emph{manifestly positive} Littlewood--Richardson rule. They are indexed by  pipe dreams $Q$ with an extra combinatorial condition, usually called the \emph{quasi-Yamanouchi pipe dreams}. This condition is similar to the Yamanouchi condition for skew Young tableaux; the precise definitions are given in \S~\ref{ssec:slidepoly}.

 Moreover, there exist combinatorial positive formulas for expressing Schubert polynomials in the slide basis: each Schubert polynomial is expressed as a linear combination of slide polynomials with coefficients 0 or 1. 

Slide polynomials also have  a $K$-theoretic counterpart: \emph{glide polynomials} $\Gc^{(\beta)}_Q$, defined by O.\,Pechenik and D.\,Searles in \cite{PechenikSearles19} (note that the names ``Schubert'' and ``Grothendieck'' also start with S and G, respectively). Similarly, there are explicit expressions of Grothendieck polynomials via glide polynomials.

\subsection{Slide complexes} The main objects defined in this paper are analogues of subword complexes corresponding to slide polynomials. We call them \emph{slide complexes}.  Each subword complex can be subdivided into slide complexes. We show that these complexes are shellable (Theorem~\ref{thm:shellable-slides}). Our main result, Theorem~\ref{thm:main}, states that each slide complex arising as a stratum in a subword complex is homeomorphic to a ball or a sphere.

In the case of pipe dream complexes, from a slide complex we can recover the corresponding slide and glide polynomials: the slide (resp. glide) polynomial is obtained as the sum of monomials corresponding to facets (resp. all interior faces) of the corresponding complex. This provides us with a topological interpretation of the combinatorial expression for $\Sf_w$ via $\Ff_Q$ and of $\Gf^{(\beta)}_w$ via $\Gc_Q^{(\beta)}$ (Corollaries~\ref{cor:glide} and~\ref{cor:slide}).

\subsection{Possible relation with degenerations of matrix Schubert varieties} In this paper we are dealing only with combinatorial constructions and do not address the geometric picture. It would be interesting to explore the relation of slide polynomials with degenerations of matrix Schubert varieties. A natural question is as follows: for a matrix Schubert variety $\overline{X_w}$, does there exist an ``intermediate degeneration'' $\overline{X_w}\to \bigcup \overline{Y_{w,Q}}$, with the irreducible components indexed by quasi-Yamanouchi pipe dreams of shape $w$, such that the multidegree of each irreducible component $\overline{Y_{w,Q}}$ is equal to the slide polynomial $\Ff_Q$? This would, in particular, provide a geometric interpretation of the Littlewood--Richardson coefficients for slide polynomials, studied by S.\,Assaf and D.\,Searles in~\cite{AssafSearles17}.


\subsection{Structure of the paper} This text is organized as follows. In Sec.\,\ref{sec:poly} we recall the definitions of Schubert and Grothendieck polynomials, provide their description using pipe dreams, and describe slide and glide polynomials in terms of pipe dreams. Sec.\,\ref{sec:subword} contains the definition of a subword complex for an arbitrary Coxeter system. We also recall the proof of its shellability and that it is homeomorphic either to a ball or to a sphere. Then we focus on the most important particular case of pipe dream complexes. The main results of this paper are contained in Sec.\,\ref{sec:sc}: in \S~\ref{ssec:sc-gen} we define a decomposition of a subword complex for an arbitrary Coxeter system into strata called slide complexes and show that these strata are shellable and homeomorphic either to balls or to spheres. In~\S~\ref{ssec:scpd} we show that the decomposition of a pipe dream complex into slide complexes corresponds to the presentation of the corresponding Schubert (resp. Grothendieck) polynomial as a sum of slide (resp. glide) polynomials. The last subsection, \S~\ref{ssec:pilaud}, describes  the relation of slide complexes with the flip graphs considered in~\cite{PilaudStump13}.

\subsection*{Acknowledgements} We are indebted to Sami Assaf, Alexander Gaifullin, Valentina Kiritchenko and Allen Knutson  for fruitful discussions. We are especially grateful to Oliver Pechenik for pointing out a simpler proof of the main result. We would like to thank the anonymous referee whose valuable remarks significantly improved the exposition. This research was supported by the HSE University Basic Research Program and by the Theoretical Physics and Mathematics Advancement Foundation ``BASIS''. E.S. was also partially supported by the RFBR grant 20-01-00091-a and the Simons--IUM Fellowship.

\section{Schubert, Grothendieck, slide and glide polynomials}\label{sec:poly}

\subsection{The symmetric group}\label{ssec:symgroup}
 We will denote by $\Sc_n$ the symmetric group on $n$ letters, i.e. the group of bijective maps from $\{1,\dots,n\}$ onto itself. It is generated by the simple transpositions $s_i=(i\leftrightarrow i+1)$ for $1\leq i\leq n-1$, modulo the Coxeter relations:
 
 \begin{itemize}
 \item $s_i^2=Id$;
 \item $s_is_j=s_js_i$ for $|i-j|\geq 2$ (far commutativity);
 \item $s_is_{i+1}s_i=s_{i+1}s_is_{i+1}$ for each $i=1,\ldots, n-2$ (braid relation).
 \end{itemize}
 
 We will use the one-line notation for permutations: for example, $w=\overline{1423}$ brings $1$ to $1$, $2$ to $4$, $3$ to $2$, and $4$ to $3$.
 
 Each permutation $w\in \Sc_n$ can be expressed as a product $w=s_{i_1}\dots s_{i_k}$ of simple transpositions. We will say that $w$ is presented by the \emph{word} $(s_{i_1},\dots,s_{i_k})$. The minimal length of a word presenting $w$ is called the \emph{length} of $w$ and denoted by $\ell(w)$. The word presenting $w$ is  said to be \emph{reduced} if its length equals $\ell(w)$. It is well known that $\ell(w)$ is equal to the number of inversions in $w$, i.e. 
 \[
 \ell(w)=\#\{(i,j)\mid 1\leq i<j\leq n, w(i)>w(j)\}.
 \]
 
The longest permutation in $\Sc_n$ will be denoted by $w_0$. This is the permutation that maps $i$ into $n+1-i$ for each $i$; clearly, $\ell(w_0)=\binom{n}{2}=\frac{n(n-1)}{2}$.

This permutation has several reduced presentations; later we will need the following one:
\[
w_0=(s_{n-1}\dots s_3s_2s_1)(s_{n-1}\dots s_3s_2)(s_{n-1}\dots s_3)\dots(s_{n-1}s_{n-2})(s_{n-1}).
\]

\subsection{Schubert and Grothendieck polynomials}\label{ssec:schubpoly}

Denote the set of variables $x_1,\ldots, x_n$ by~$\xx$ and consider the polynomial ring $\ZZ[\xx]$. The group $\Sc_n$ acts on this ring by interchanging variables:
\[
w\circ f(x_1,\ldots, x_n)=f(x_{w(1)},\ldots,x_{w(n)}).
\]
\begin{definition}
	For $i=1,\ldots,n-1$, we define the \emph{divided difference operators} $\partial_i:\ZZ[\xx]\to\ZZ[\xx]$ as follows:
\[
\partial_i f(\xx)=\frac{f(\xx)-s_i\circ f(\xx)}{x_i-x_{i+1}}.
\]
\end{definition}	
Since the numerator is antisymmetric with respect to  $x_i$ and $x_{i+1}$, it is divisible by the denominator, so the ratio is indeed a polynomial with integer coefficients.

The divided difference operators satisfy the Coxeter relations:
\begin{itemize}
	\item $\partial_i^2=0$,
	\item $\partial_{i}\partial_j=\partial_j\partial_i$  if $|i-j|\geq 2$,
	\item $\partial_{i} \partial_{i+1} \partial_{i}=\partial_{i+1} \partial_{i} \partial_{i+1}$ for each $i=1,\ldots, n-2$.
\end{itemize}

\begin{definition}
	\emph{Schubert polynomials} $\Sf_w$ are defined as the elements of $\ZZ[\xx]$ indexed by  permutations $w\in\Sc_n$  and satisfying the relations
		$$
	\mathfrak S_{Id}=1,
	$$
	$$
	\partial_{i} \mathfrak S_{w}=\begin{cases}
	\mathfrak{S}_{w s_{i}}, & \text { if } \ell\left(w s_{i}\right)<\ell(w), \\
	0, & \text { otherwise}
	\end{cases}
	$$
	for each $i=1,\ldots, n-1$. 
\end{definition}
 	
A.\,Lascoux and M.-P.\,Schützenberger \cite{LascouxSchutzenberger82} have shown that the Schubert polynomials are uniquely determined by these relations. Equivalently, they can be constructed by using the recurrence relation
\[
 	\mathfrak{S}_{w s_{i}}(\mathbf{x})=\partial_{i} \mathfrak{S}_{w}(\mathbf{x}), \text{ if } \ell(ws_i)<\ell(w),
\]
with the initial condition
\[
\mathfrak{S}_{w_0}(\mathbf x)=x_1^{n-1}x_2^{n-2}\ldots x_{n-2}^2x_{n-1}.
\]
This recurrence relation can be written as follows: if $s_{i_k}\ldots s_{i_1}$ is a reduced word for a permutation $w_0w$, then
\[
\mathfrak{S}_w=\partial_{i_1}\ldots\partial_{i_k}\mathfrak{S}_{w_0}.
\]
Since the divided difference operators satisfy the far commutativity and braid relations, and every reduced word for $w_0w$ can be transformed into any other reduced word just by these two operations, $\Sf_w$ is well defined (i.e. does not depend upon a choice of the reduced word).

Grothendieck polynomials were introduced by A.\,Lascoux in~\cite{Lascoux90}. We will use their deformation, $\beta$-Grothendieck polynomials, introduced by S.\,Fomin and An.\,Kirillov in~\cite{FominKirillov94}. Sometimes we will refer to them simply as to Grothendieck polynomials. Their definition is similar to the definition of Schubert polynomials, but instead of $\partial_i$ we need to use the \emph{isobaric divided difference operators} $\pi^{(\beta)}_i$.

\begin{definition}
Let $\beta$ be a formal parameter. For $i=1,\ldots,n-1$ define the \emph{$\beta$-isobaric divided difference operators} $\pi^{(\beta)}_i:\ZZ[\beta,\xx]\to\ZZ[\beta,\xx]$:
\[
\pi^{(\beta)}_if(\xx)=\frac{(1+\beta x_{i+1})f(\xx)-(1+\beta x_i)s_i\circ f(\xx)}{x_i-x_{i+1}}.
\]
\end{definition}
Just like the divided difference operators, their isobaric counterparts also satisfy the Coxeter relations:
	\begin{itemize}
		\item $\pi^{(\beta)}_{i}\pi^{(\beta)}_j=\pi^{(\beta)}_j\pi^{(\beta)}_i$ for $|i-j|\geq 2$,
		\item $\pi^{(\beta)}_{i} \pi^{(\beta)}_{i+1} \pi^{(\beta)}_{i}=\pi^{(\beta)}_{i+1} \pi^{(\beta)}_{i} \pi^{(\beta)}_{i+1}$ for each $i=1,\ldots, n-2$.
	\end{itemize}
\begin{definition}
Define \emph{$\beta$-Grothendieck polynomials} $\Gf^{(\beta)}_w$ using the initial condition
		$$\Gf^{(\beta)}_{w_0}(\mathbf x)=x_1^{n-1}x_2^{n-2}\ldots x_{n-2}^2x_{n-1}$$
and the recurrence relation
\[
\Gf^{(\beta)}_w=\pi^{(\beta)}_{i_1}\ldots\pi^{(\beta)}_{i_k}\Gf^{(\beta)}_{w_0},
\]
where $s_{i_k}\ldots s_{i_1}$ is a reduced word for the permutation $w_0w$.
\end{definition}
Since the operators $\pi^{(\beta)}_i$ satisfy the Coxeter relations, these polynomials are also well-defined. One can immediately see that, since $\pi^{(0)}_i=\partial_i$ and $\Gf^{(\beta)}_{w_0}=\Sf_{w_0}$, we have $\Gf^{(0)}_w=\Sf_w$ for each $w\in\Sc_n$.	So setting in $\Gf^{(\beta)}$ the parameter $\beta=0$, we recover the Schubert polynomials.

\subsection{Pipe dreams}\label{ssec:pipedreams}

In this subsection we discuss pipe dreams: the main combinatorial tool for dealing with Schubert and Grothendieck polynomials. 

\begin{definition} 
Consider an $(n\times n)$-square and fill it with the elements of two types:
\emph{crosses} $\cross$ and \emph{elbows} $\elbows$ in such a way that all the crosses are situated strictly above the antidiagonal. We will omit the elbows situated below the antidiagonal. This diagram is called a \emph{pipe dream}, or an \emph{rc-graph} (``RC'' stands for ``reduced compatible'').
	
Each pipe dream can be viewed as a configuration of $n$ strands joining the left edge of the square with the top edge. Let us index the initial and terminal points of these strands by the numbers from 1 to $n$, going from top to bottom and from left to right.
	
	Pipe dream is said to be  \emph{reduced} if every pair of strands crosses at most once and \emph{nonreduced} otherwise.
	
\end{definition}	

The following figure provides an example of reduced and non-reduced pipe dreams.
	\begin{figure}[ht]
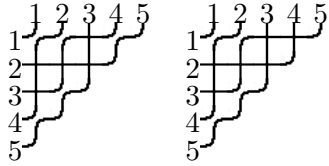

		$$
		\begingroup
		\setlength\arraycolsep{0pt}
		\renewcommand{\arraystretch}{0.07}
		\begin{matrix}
		& 1 & 2 & 3 & 4 & 5  \\
		1 & \elbows & \elbows & \cross & \elbows & \elbow& \\
		2 & \cross & \cross & \cross & \elbow& \\
		3 & \cross & \elbows & \elbow&  \\
		4 & \elbows & \elbow&\\
		5 & \elbow& \\
		\end{matrix}
		\quad
		\begin{matrix}
		& 1 & 2 & 3 & 4 & 5  \\
		1 & \elbows & \elbows & \cross & \cross & \elbow& \\
		2 & \cross & \cross & \cross & \elbow& \\
		3 & \cross & \elbows & \elbow& \\
		4 & \elbows & \elbow& \\
		5 & \elbow& \\
		\end{matrix}
		\endgroup
		$$
		\caption{\label{pipedreams} A reduced and a non-reduced pipe dream}
	\end{figure}
\begin{definition}

Each pipe dream can be viewed as a bijective map from the set of the initial points of the strands to the set of its terminal points. Let us assign to each \emph{reduced} pipe dream $P$ the corresponding permutation $w(P)\in\Sc_n$. It will be called the \emph{shape} of $P$. Moreover, to each pipe dream $P$ we associate the set  $D_P$ of the coordinates of its crosses (the first and the second coordinates stand for the row and the column numbers, respectively).
\end{definition} 
For example, the shape of the left pipe dream $P$ at Fig.\,\ref{pipedreams} equals $w(P)=\overline{15423}$, and the set $D_P$ is equal to $ D_P=\{(1,3),(2,1),(2,2),(2,3),(3,1)\}$.

\begin{definition}
Define the \emph{reduction} operation  $\mathrm{reduct}$ on the set of pipe dreams as follows: we read the rows of a pipe dream from top to bottom, reading each row from right to left. Each time we find a crossing of two strands that have already crossed before (i.e. above), we replace this cross by an elbow.
\end{definition}

\begin{figure}[ht]
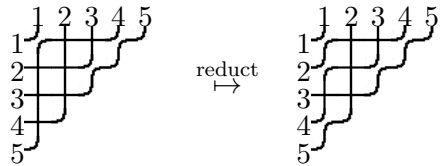

	$$
	\begingroup
	\setlength\arraycolsep{0pt}
	\renewcommand{\arraystretch}{0.07}
	\quad\begin{matrix}
	& 1 & 2 & 3 & 4 & 5  \\
	1 & \elbows & \cross & \cross & \elbows & \elbow & \\
	2 & \cross & \cross & \elbows & \elbow & \\
	3 & \cross & \cross & \elbow & \\
	4 & \cross & \elbow & \\
	5 & \elbow & \\
	\end{matrix}
	\quad
	\overset{\mathrm{reduct}}{\mapsto}
	\quad
	\begin{matrix}
	& 1 & 2 & 3 & 4 & 5  \\
	1 & \elbows & \cross & \cross & \elbows & \elbow & \\
	2 & \elbows & \cross & \elbows & \elbow & \\
	3 & \cross & \cross & \elbow & \\
	4 & \elbows & \elbow & \\
	5 & \elbow & \\
	\end{matrix}
	\endgroup
	$$
	\caption{\label{reduct} The reduction operation $\mathrm{reduct}$ applied to a non-reduced  pipe dream}

\end{figure}

Obviously, $\mathrm{reduct}(P)$ is reduced for each $P$, and the operation acts trivially on reduced pipe dreams: $\mathrm{reduct}(P)=P$. We say that the \emph{shape} of a nonreduced pipe dream $P$ is the shape $w(\mathrm{reduct}(P))$ of its reduction.

The set of all pipe dreams of a given shape $w\in\Sc_n$ will be denoted by $\PD(w)$. We will also denote the subset of reduced pipe dreams of this shape by  $\PD_0(w)\subset\PD(w)$.

The following theorem was  proved by S.\,Billey and N.\,Bergeron and independently by S.\,Fomin and An.\,Kirillov.

\begin{theorem}\cite{BilleyBergeron93,FominKirillov96}\label{thm:FK}
The Schubert polynomials satisfy the equality
\[
\mathfrak S_w=\sum_{P \in \PD_0(w)}\mathbf x^P,
\]
where
\[
\mathbf x^P:=\prod_{(i,j)\in D_P}x_i.
\]
\end{theorem}

\begin{example}\label{ex:1432}
	Consider the permutation $w=\overline{1432}$.  There are five reduced pipe dreams of 
	shape~$w$:
\[
	\begingroup
	\setlength\arraycolsep{0pt}
	\renewcommand{\arraystretch}{0.06}
	\begin{matrix}
	& 1 & 2 & 3 & 4  \\
	1 & \elbows & \cross & \cross & \elbow & \\
	2 & \elbows & \cross & \elbow & \\
	3 & \elbows & \elbow & \\
	4 & \elbow & \\
	\end{matrix}\quad
	\begin{matrix}
	& 1 & 2 & 3 & 4  \\
	1 & \elbows & \cross & \cross & \elbow & \\
	2 & \elbows & \elbows & \elbow & \\
	3 & \cross & \elbow & \\
	4 & \elbow & \\
	\end{matrix}
	\quad\begin{matrix}
	& 1 & 2 & 3 & 4  \\
	1 & \elbows & \cross & \elbows & \elbow & \\
	2 & \cross & \cross & \elbow & \\
	3 & \elbows & \elbow & \\
	4 & \elbow & \\
	\end{matrix}
	\quad\begin{matrix}
	& 1 & 2 & 3 & 4  \\
	1 & \elbows & \elbows & \cross & \elbow & \\
	2 & \cross & \elbows & \elbow & \\
	3 & \cross & \elbow & \\
	4 & \elbow & \\
	\end{matrix}
	\quad\begin{matrix}
	& 1 & 2 & 3 & 4  \\
	1 & \elbows & \elbows & \elbows & \elbow & \\
	2 & \cross & \cross & \elbow & \\
	3 & \cross & \elbow & \\
	4 & \elbow & \\
	\end{matrix}
	\quad
	\endgroup
\]
Hence the Schubert polynomial for $w$ looks as follows:
\[
\mathfrak S_{\overline{1432}}(x_1,x_2,x_3)=x_1^2x_2+x_1^2x_3+x_1x_2^2+x_1x_2x_3+x_2^2x_3.
\]
\end{example}

This theorem has the following modification for Grothendieck polynomials.

\begin{definition}
Denote by $\mathrm{ex}(P)$ the \emph{excess} of $P$, i.e. the number of ``redundant'' crosses in a (non-reduced) pipe dream $P$. Namely, set $\mathrm{ex}(P)=\#\left(D_P\smallsetminus D_{\mathrm{reduct}(P)}\right)$.
\end{definition}
\begin{theorem}\cite{FominKirillov93,FominKirillov94}\label{thm:gFK}
The Grothendieck polynomials satisfy the following identity:
\[
\Gf^{(\beta)}_w=\sum_{P\in\PD(w)}\beta^{\mathrm{ex}(P)}\xx^P.
\]
\end{theorem}
\begin{example}
To continue Example~\ref{ex:1432}, let us compute the Grothendieck polynomial for $w=\overline{1432}$. The set $\PD(w)$ consists of 11 pipe dreams, 5 of them being reduced and 6 non-reduced.
\[
	\begingroup
	\setlength\arraycolsep{0pt}
	\renewcommand{\arraystretch}{0.07}
	\begin{matrix}
& 1 & 2 & 3 & 4  \\
1 & \elbows & \cross & \cross & \elbow & \\
2 & \elbows & \cross & \elbow & \\
3 & \elbows & \elbow & \\
4 & \elbow & \\
\end{matrix}
\quad	
	\begin{matrix}
	& 1 & 2 & 3 & 4  \\
	1 & \elbows & \cross & \cross & \elbow & \\
	2 & \elbows & \elbows & \elbow & \\
	3 & \cross & \elbow & \\
	4 & \elbow & \\
	\end{matrix}
	\quad
	\begin{matrix}
	& 1 & 2 & 3 & 4  \\
	1 & \elbows & \cross & \elbows & \elbow & \\
	2 & \cross & \cross & \elbow & \\
	3 & \elbows & \elbow & \\
	4 & \elbow & \\
	\end{matrix}
	\quad
	\begin{matrix}
	& 1 & 2 & 3 & 4  \\
	1 & \elbows & \elbows & \cross & \elbow & \\
	2 & \cross & \elbows & \elbow & \\
	3 & \cross & \elbow & \\
	4 & \elbow & \\
	\end{matrix}
	\quad
	\begin{matrix}
	& 1 & 2 & 3 & 4  \\
	1 & \elbows & \elbows & \elbows & \elbow & \\
	2 & \cross & \cross & \elbow & \\
	3 & \cross & \elbow & \\
	4 & \elbow & \\
	\end{matrix}
	\quad
	\endgroup
\]
	
\[
	\begingroup
	\setlength\arraycolsep{0pt}
	\renewcommand{\arraystretch}{0.07}
		\begin{matrix}
	& 1 & 2 & 3 & 4  \\
	1 & \elbows & \cross & \cross & \elbow & \\
	2 & \cross & \cross & \elbow & \\
	3 & \elbows & \elbow & \\
	4 & \elbow & \\
	\end{matrix}
	\quad
		\begin{matrix}
	& 1 & 2 & 3 & 4  \\
	1 & \elbows & \cross & \cross & \elbow & \\
	2 & \cross & \elbows & \elbow & \\
	3 & \cross & \elbow & \\
	4 & \elbow & \\
	\end{matrix}
	\quad
	\begin{matrix}
	& 1 & 2 & 3 & 4  \\
	1 & \elbows & \cross & \cross & \elbow & \\
	2 & \elbows & \cross & \elbow & \\
	3 & \cross & \elbow & \\
	4 & \elbow & \\
	\end{matrix}
	\quad
	\begin{matrix}
	& 1 & 2 & 3 & 4  \\
	1 & \elbows & \cross & \elbows & \elbow & \\
	2 & \cross & \cross & \elbow & \\
	3 & \cross & \elbow & \\
	4 & \elbow & \\
	\end{matrix}
	\quad
	\begin{matrix}
	& 1 & 2 & 3 & 4  \\
	1 & \elbows & \elbows & \cross & \elbow & \\
	2 & \cross & \cross & \elbow & \\
	3 & \cross & \elbow & \\
	4 & \elbow & \\
	\end{matrix}
	\quad
	\endgroup
\]
\[
	\begingroup
	\setlength\arraycolsep{0pt}
	\renewcommand{\arraystretch}{0.07}
	\begin{matrix}
	& 1 & 2 & 3 & 4  \\
	1 & \elbows & \cross & \cross & \elbow & \\
	2 & \cross & \cross & \elbow & \\
	3 & \cross & \elbow & \\
	4 & \elbow & \\
	\end{matrix}
	\quad
	\endgroup
\]
The corresponding Grothendieck polynomial is equal to
\[
\Gf^{(\beta)}_{(1432)} =x_{1}^{2}x_{2}+x_{1}^{2}x_{3}+x_{1}x_{2}^{2}+x_{1}x_{2}x_{3}+x_{2}^{2}x_{3}+ \beta x_{1}^{2}x_{2}^{2}+2\beta x_{1}^{2}x_{2}x_{3}+2\beta x_{1}x_{2}^{2}x_{3}+\beta^{2}x_{1}^{2}x_{2}^{2}x_{3}.
\]
\end{example}
Since $\mathrm{ex}(P)=0$ if and only if $P$ is reduced, we have $\Gf_w^{(\beta)}=\Sf_w+\beta(\ldots)$. This implies the equality $\Gf_w^{(0)}=\Sf_w$ we mentioned above.

\subsection{Slide and glide polynomials}\label{ssec:slidepoly}

S.\,Assaf and D.\,Searles~ \cite{AssafSearles17}  introduced another basis in the ring of polynomials: the \emph{slide polynomials}. One of their main features is that each Schubert polynomial can be represented as a sum of slide polynomials with the coefficients 0 or 1. The subsequent paper by O.\,Pechenik and D.\,Searles~\cite{PechenikSearles19} provides a similar construction for Grothendieck polynomials. Let us recall these constructions here.

\begin{definition}
Let $P$ be a (possibly non-reduced) pipe dream. Denote a \emph{slide move} $S_i$ as follows. Suppose that the leftmost  cross in the $i$-th row of $P$ is located strictly to the right of the rightmost cross in the $(i+1)$-st row (in particular, the row $i+1$ can contain only elbows). In this case, the leftmost cross in the $i$-th row can be shifted one step southwest: $\begingroup
	\setlength\arraycolsep{0pt}
	\renewcommand{\arraystretch}{0.07}
	\begin{matrix}
	\elbows & \cross\\
	\elbows & \elbows
	\end{matrix} \mapsto
	\begin{matrix}
	\elbows & \elbows\\
	\cross & \elbows
	\end{matrix}
	\endgroup $. If the initial pipe dream was non-reduced, this move can look as follows: $\begingroup
	\setlength\arraycolsep{0pt}
	\renewcommand{\arraystretch}{0.07}
	\begin{matrix}
	\elbows & \cross\\
	\cross & \elbows
	\end{matrix} \mapsto
	\begin{matrix}
	\elbows & \elbows\\
	\cross & \elbows
	\end{matrix}
	\endgroup $. If the leftmost cross in the $i$-th row is either in the first column or weakly left of a cross in the $(i+1)$-st row, we will say that $S_i$ acts on $P$ identically.
	\end{definition}

Note that a slide move preserves the shape of a pipe dream. Indeed,  $\mathrm{reduct}(S_i(P))$ and  $\mathrm{reduct}(P)$ either coincide or are obtained one from the other by a shape-preserving move of one cross: $\begingroup
\setlength\arraycolsep{0pt}
\renewcommand{\arraystretch}{0.07}
\begin{matrix}
	\elbows & \cross\\
	\elbows & \elbows
\end{matrix} \mapsto
\begin{matrix}
	\elbows & \elbows\\
	\cross & \elbows
\end{matrix}
\endgroup $.
Moreover, the number of crosses in a reduced pipe dream is preserved by a slide move. This means that a slide move sends a reduced pipe dreams to a reduced one.

\begin{definition}\label{def:quasiyamanouchi} If all slide moves act on $P$ identically: i.e., for each $i$ we either have the $i$-th row starting with a cross, or the leftmost cross in the $i$-th row is located weakly left of a cross from the $(i+1)$-st row, then $P$ is said to be \emph{quasi-Yamanouchi}.

Denote the set of all quasi-Yamanouchi pipe dreams of shape  $w$ by $\QPD(w)\subset\PD(w)$, and the subset of all reduced quasi-Yamanouchi pipe dreams by $\QPD_0(w)=\PD_0(w)\cap\QPD(w)$.
\end{definition}

\begin{definition}\label{def:destand}
The \emph{destandardization} operations $\dst\colon\PD(w)\to\QPD(w)$ and $\dst_0\colon\PD_0(w)\to\QPD_0(w)$ are defined as repeated applications of slide moves to a pipe dream until it becomes quasi-Yamanouchi.
\end{definition}

In~\cite[Lemma~3.12]{AssafSearles17} it is shown that every pipe dream can be sent into a quasi-Yamanouchi one by repeated applications of slide moves, and that the resulting quasi-Yamanouchi pipe dream does not depend on the order of slide moves and hence is well-defined. Both operations $\dst\colon\PD(w)\to\QPD(w)$ and $\dst_0\colon\PD_0(w)\to\QPD_0(w)$ are obviously surjective, since they are projectors to the sets of quasi-Yamanouchi and reduced quasi-Yamanouchi pipe dreams respectively. 

\begin{definition}
Let $Q\in\QPD_0(w)$ be a reduced quasi-Yamanouchi pipe dream. The set $\dst_0^{-1}(Q)$ is called the \emph{slide orbit} of $Q$. If  $Q\in\QPD(w)$ is a not necessarily reduced quasi-Yamanouchi pipe dream, then $\dst^{-1}(Q)$ is called the \emph{glide orbit} of $Q$.
\end{definition}

\begin{definition}
For $Q\in\QPD_0(w)$, the \emph{slide polynomial} $\Ff_Q$ is defined as the sum of monomials over the corresponding slide orbit of pipe dreams:
\[
\Ff_Q=\sum_{P\in\dst_0^{-1}(Q)}\mathbf x^P.
\]
For $Q\in\QPD(w)$, the \emph{glide polynomial} $\Gc^{(\beta)}_Q$ is defined as the sum of monomials over the glide orbit of pipe dreams:
	$$\Gc^{(\beta)}_Q=\sum_{P\in\dst^{-1}(Q)}\beta^{\mathrm{ex}(P)-\mathrm{ex}(Q)}\mathbf x^P.$$
\end{definition}

This definition together with Theorems~\ref{thm:FK} and~\ref{thm:gFK} implies that
\[
\Sf_w=\sum_{Q\in\QPD_0(w)} \Ff_Q\qquad\text{and}\qquad\Gf^{(\beta)}_w=\sum_{Q\in\QPD(w)}\beta^{\mathrm{ex}(Q)}\Gc^{(\beta)}_Q.
\]

\begin{example}
There are five quasi-Yamanouchi pipe dreams of shape $w=\overline{1432}$. This means that  $\PD(w)$ splits into 5 glide orbits. One of them consists of seven pipe dreams (the quasi-Yamanouchi pipe dream is shown in parentheses):
\[
	\begingroup
	\setlength\arraycolsep{0pt}
	\renewcommand{\arraystretch}{0.07}
	\begin{pmatrix}
		& 1 & 2 & 3 & 4  \\
		1 & \elbows & \elbows & \elbows & \elbow & \\
		2 & \cross & \cross & \elbow & \\
		3 & \cross & \elbow & \\
		4 & \elbow & \\
	\end{pmatrix}
	\quad
	\begin{matrix}
		& 1 & 2 & 3 & 4  \\
		1 & \elbows & \elbows & \cross & \elbow & \\
		2 & \cross & \elbows & \elbow & \\
		3 & \cross & \elbow & \\
		4 & \elbow & \\
	\end{matrix}
	\quad
	\begin{matrix}
		& 1 & 2 & 3 & 4  \\
		1 & \elbows & \elbows & \cross & \elbow & \\
		2 & \cross & \cross & \elbow & \\
		3 & \cross & \elbow & \\
		4 & \elbow & \\
	\end{matrix}
	\quad
	\begin{matrix}
		& 1 & 2 & 3 & 4  \\
		1 & \elbows & \cross & \cross & \elbow & \\
		2 & \elbows & \elbows & \elbow & \\
		3 & \cross & \elbow & \\
		4 & \elbow & \\
	\end{matrix}
	\endgroup
\]
\[
	\begingroup
	\setlength\arraycolsep{0pt}
	\renewcommand{\arraystretch}{0.1}
	\begin{matrix}
		& 1 & 2 & 3 & 4  \\
		1 & \elbows & \cross & \cross & \elbow & \\
		2 & \elbows & \cross & \elbow & \\
		3 & \elbows & \elbow & \\
		4 & \elbow & \\
	\end{matrix}
	\quad
	\begin{matrix}
		& 1 & 2 & 3 & 4  \\
		1 & \elbows & \cross & \cross & \elbow & \\
		2 & \cross & \elbows & \elbow & \\
		3 & \cross & \elbow & \\
		4 & \elbow & \\
	\end{matrix}
	\quad
	\begin{matrix}
		& 1 & 2 & 3 & 4  \\
		1 & \elbows & \cross & \cross & \elbow & \\
		2 & \elbows & \cross & \elbow & \\
		3 & \cross & \elbow & \\
		4 & \elbow & \\
	\end{matrix}
	\quad
	\endgroup
\]
Consequently, the corresponding glide polynomial equals
\[
\Gc^{(\beta)}_{s_3s_2s_3}= 2\beta x_{1}^{2}x_{2}x_{3}+\beta x_{1}x_{2}^{2}x_{3}+x_{1}^{2}x_{2}+x_{1}^{2}x_{3}+x_{1}x_{2}x_{3}+x_{2}^{2}x_{3}.
\]
Each of the remaining four glide orbits consists of one pipe dream:
\[
\begingroup
\setlength\arraycolsep{0pt}
\renewcommand{\arraystretch}{0.07}
\begin{pmatrix}
& 1 & 2 & 3 & 4  \\
1 & \elbows & \cross & \elbows & \elbow & \\
2 & \cross & \cross & \elbow & \\
3 & \cross & \elbow & \\
4 & \elbow & \\
\end{pmatrix}
\quad
\begin{pmatrix}
& 1 & 2 & 3 & 4  \\
1 & \elbows & \cross & \cross & \elbow & \\
2 & \cross & \cross & \elbow & \\
3 & \elbows & \elbow & \\
4 & \elbow & \\
\end{pmatrix}
\quad
\begin{pmatrix}
& 1 & 2 & 3 & 4  \\
1 & \elbows & \cross & \cross & \elbow & \\
2 & \cross & \cross & \elbow & \\
3 & \cross & \elbow & \\
4 & \elbow & \\
\end{pmatrix}
\quad
\begin{pmatrix}
& 1 & 2 & 3 & 4  \\
1 & \elbows & \cross & \elbows & \elbow & \\
2 & \cross & \cross & \elbow & \\
3 & \elbows & \elbow & \\
4 & \elbow & \\
\end{pmatrix}
\quad
\endgroup
\]
Hence each of the corresponding glide polynomials has only one term:
\[
\Gc^{(\beta)}_{s_2s_3s_2s_3}=x_1x_2^2x_3;\qquad\Gc^{(\beta)}_{s_3s_2s_3s_2}=x_1^2x_2^2;\qquad\Gc^{(\beta)}_{s_3s_2s_3s_2s_3}=x_1^2x_2^2x_3;\qquad\Gc^{(\beta)}_{s_2s_3s_2}=x_1x_2^2.
\]

\end{example}

\section{Subword complexes and pipe dream  complexes}\label{sec:subword}

\subsection{Subword complexes}\label{ssec:subword}

Consider an arbitrary Coxeter system $(\Pi, \Sigma)$, where $\Pi$ is a Coxeter group, and $\Sigma$ a system of simple reflections minimally generating $\Pi$. We will be particularly interested in the situation where $\Pi=\Sc_n$ is a symmetric group and $\Sigma=\{s_1,\dots,s_{n-1}\}$ is the set of simple transpositions.

\begin{definition}
	A \emph{word} of length $m$ is a sequence $\cQ=(\sigma_1, \ldots, \sigma_m)$ of simple reflections. A subsequence $\cP=(\sigma_{i_1},\sigma_{i_2}, \ldots, \sigma_{i_k})$, where $1\leqslant i_1<i_2<\ldots<i_k\leqslant m$, is a \emph{subword} of $\cQ$.
	
	We say that $\cP$ \emph{represents} $\pi \in \Pi$ if   $\sigma_{i_1} \sigma_{i_2} \ldots \sigma_{i_k}$ is a reduced decomposition of $\pi$. If some subword of $\cP$ represents $\pi$, we will say that  $\cP$ \emph{contains} $\pi$.

	The \emph{subword complex} $\Delta(\cQ, \pi)$ is a set of nonempty subwords $\cQ\smallsetminus \cP$ whose complements $\cP$ contain $\pi$.  It is a simplicial complex; one of its faces belongs to the boundary of another if and only if the first of the corresponding subwords is a subset of the second one.
\end{definition}

All reduced subwords for $\pi \in \Pi$ have the same length. So the complex $\Delta(\cQ, \pi)$ is pure of dimension $m-\ell(\pi)$. The words $\cQ \smallsetminus \cP$ such that $\cP$ represents $\pi$ are its facets.

\begin{remark} To distinguish between words in an arbitrary Coxeter system and pipe dreams, we use the calligraphic font, such as in $\cP$, $\cQ$ for words and the regular font, such as in $P$, $Q$, for pipe dreams.
\end{remark}

\begin{example} 
	Let $\Pi=\Sc_4, \pi=\overline{1432}, \cQ=s_3s_2s_1s_3s_2s_3$. The permutation $\pi$ has two reduced decompositions $s_2s_3s_2$ and $s_3s_2s_3$. Let us label the center of a pentagon  with $s_1$ and its vertices with reflections $s_3,s_2,s_3,s_2,s_3$ in the cyclic order. Then the facets of $\Delta(\cQ, \pi)$ are the triples that consist of two adjacent vertices and the center of the pentagon.
	\begin{figure}[ht]
	\centering
	\includegraphics[width=5cm]{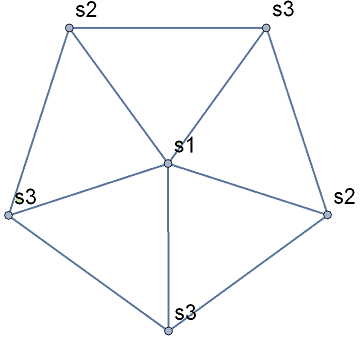}
	\caption{\label{complex-example} Subword complex $\Delta(s_3s_2s_1s_3s_2s_3, s_2s_3s_2)$.}
	\end{figure}
\end{example}	

\begin{definition}
	Let $\Delta$ be a simplicial complex and $F\in\Delta$ a face. The \emph{deletion} of $F$ from $\Delta$ is the complex 
\[
\del(F, \Delta)=\{G \in \Delta \mid G \cap F = \varnothing\}.
\]
	The \emph{link} of $F$ in $\Delta$ is the complex
\[
\link(F, \Delta)=\{G \in \Delta \mid G \cap F = \varnothing, G \cup F \in \Delta\}.
\]
\end{definition}

\begin{definition} 
An $n$-dimensional complex $\Delta$ is said to be \emph{vertex decomposable} if it is pure and satisfies one of the following properties:
	\begin{itemize}
		\item $\Delta$ is an $n$-dimensional simplex; or
		\item there exists a vertex $v \in \Delta$ such that  $\del(v, \Delta)$ is a vertex decomposable $n$-dimensional complex, while $\link(v, \Delta)$ is an $(n-1)$-dimensional vertex decomposable complex.
	\end{itemize}
\end{definition}	

\begin{definition}\label{def:shelling}
	A \emph{shelling} of a simplicial complex  $\Delta$ is a total order on the set of its facets with the following property: for each  $i,j$, such that  $1 \leqslant i < j \leqslant t$, there exist $k$, where $1 \leqslant k < j$, and a vertex $v \in F_j$ such that $F_i \cap F_j \subseteq F_k \cap F_j = F_j \smallsetminus \{v\}$.
	
A complex admitting a shelling is said to be \emph{shellable}.
\end{definition}

The definition of shelling can be restated in the following way: for each $j$ such that $2\leqslant j \leqslant t$, the complex  $(\bigcup\limits_{i<j}F_i)\cap F_j$ is pure of dimension $\dim F_j-1$.

The notion of vertex decomposability was introduced in~\cite{vertex-dec}; in the same paper it is shown that it implies the shellability.
\begin{proposition}[\cite{vertex-dec}]
Vertex decomposable complexes are shellable. 
\end{proposition}

The following statement was proved in~\cite[Thm~2.5]{KnutsonMiller04}.

\begin{theorem}\label{thm:subword-shellable} 
Subword complexes  are vertex decomposable, and hence shellable.
\end{theorem}

\begin{definition} The \emph{Demazure product} $\delta(\cQ)\in\Pi$ of a word $\cQ$ is defined by induction as follows: $\delta(\sigma)=\sigma$ for $\sigma\in \Sigma$, and 
\[
\delta(\cQ,\sigma)=\begin{cases}
\delta(\cQ)\sigma, &\text{if }\ell(\delta(\cQ)\sigma)>\ell(\delta(\cQ)),\\
\delta(\cQ), &\text{otherwise}.\end{cases}
\]
In other words, we multiply the elements in $\cQ$ from left to right, omitting the letters that decrease the length of the product obtained at each step. One can also think about the Demazure product as the product in the monoid generated by $\Sigma$ subject to the relations of the Coxeter group with the relation $s_i^2=e$ being replaced by $s_i^2=s_i$, cf.\,\cite[Def.\,3.1]{KnutsonMiller04}.
\end{definition}

Now let us recall the main results from~\cite{KnutsonMiller04}.
\begin{theorem}[{\cite[Thm~3.7]{KnutsonMiller04}}]\label{thm:ball-or-sphere}
The complex $\Delta(\cQ, \pi)$ is either a ball or a sphere. A face $\cQ\smallsetminus \cP$ is contained in its boundary iff $\delta(\cP)\ne \pi$.
\end{theorem}

\begin{corollary}[{\cite[Cor.\,3.8]{KnutsonMiller04}}]	  $\Delta(\cQ, \pi)$ is a sphere if $\delta(\cQ)=\pi$, and a ball otherwise.
\end{corollary}

\subsection{Pipe dream complexes}\label{ssec:pdc}

Pipe dreams are closely related to subword complexes of a certain form. Let $\Pi=\Sc_n$, and let us fix the following word for the longest permutation:
\[
\cQ_{0,n}=(s_{n-1} s_{n-2}\ldots s_3 s_2 s_1) (s_{n-1} s_{n-2} \ldots s_3 s_2) (s_{n-1} s_{n-2} \ldots s_3) \ldots( s_{n-1} s_{n-2}) (s_{n-1}).
\]
This word is obtained by reading the table
		$$
	\begin{matrix}
	s_1 & s_2 & s_3 &\dots & s_{n-2}& s_{n-1}&\\
	s_2 & s_3& \dots &s_{n-2}&s_{n-1}&\\
	s_3&  \dots & s_{n-2} &s_{n-1}\\
	\vdots &\vdots &\iddots\\
	s_{n-2} &s_{n-1}\\
	s_{n-1}
	\end{matrix}
	$$
from right to left, from top to bottom. Let $P\in\PD_0(w)$ be a reduced pipe dream of shape $w\in\Sc_n$. For each of the crosses occuring in $P$, let us take the simple transposition from the corresponding cell of the table. We obtain a subword $\word(P)$ in $\cQ_{0,n}$. It is clear that  $\word(P)$ represents the permutation  $w$. The converse is also true: if $\cT$ is a subword in $\cQ_{0,n}$ representing  $w$, then the pipe dream with crosses corresponding to the letters of $\cT$ is reduced and has the shape $w$.

We obtain a bijection betwen the elements of $\PD_0(w)$ and facets of the complex $\Delta(\cQ_{0,n},w)$.
\begin{definition}
The complex $\Delta(\cQ_{0,n},w)$ is called a \emph{pipe dream complex}.
\end{definition}
The reduction of a pipe dream corresponds naturally to the  Demazure product of the related subword: for each pipe dream $P$, we have $w(\mathrm{reduct}(P))=\delta(\mathrm{word}(P))$. Moreover, if $P$ is a nonreduced pipe dream of shape $w$, then $\mathrm{word}(P)$ contains  $\mathrm{word}(\mathrm{reduct}(P))$ as a subword and hence contains the permutation $w$.

Using these facts and Theorem~\ref{thm:ball-or-sphere}, we see that the pipe dreams from $\PD(w)$ bijectively correspond to the interior faces of the pipe dream complex $\Delta(\cQ_{0,n},w)$.

The following description of Grothendieck polynomials in terms of pipe dream complexes is essentially due to A.\,Knutson and E.\,Miller, cf.~\cite[Cor.~5.5]{KnutsonMiller04}. Sometimes, as in~\cite{EscobarMeszaros18}, it is used as an equivalent definition of Grothendieck polynomials.

\begin{corollary}\label{groth-sum}
The Grothendieck polynomial $\Gf_w^{\beta}$ is obtained as the sum of monomials corresponding to the interior faces of the corresponding pipe dream complex. Namely, for $w\in\Sc_n$, we have
 	$$\Gf_w^{(\beta)}=\sum_{P\in\mathrm{int}(\Delta(Q_{0,n}, w))} \beta^{\mathrm{codim}(P)}\xx^P$$
 	(By $\xx^P$ we denote the monomial of the pipe dream corresponding to a face  $P$).
\end{corollary}
\begin{example}\label{ex:pentagon}
Fig.\,\ref{pd-complex} represents the pipe dream complex for $w=\overline{1432}$. The pipe dreams are split into groups corresponding to their shapes; each pipe dream is indexed by the corresponding monomial $\beta^{\mathrm{ex}(P)}\xx^P$ occuring in the Grothendieck polynomial.
	\begin{figure}[ht!]
		$$\includegraphics[width=17cm]{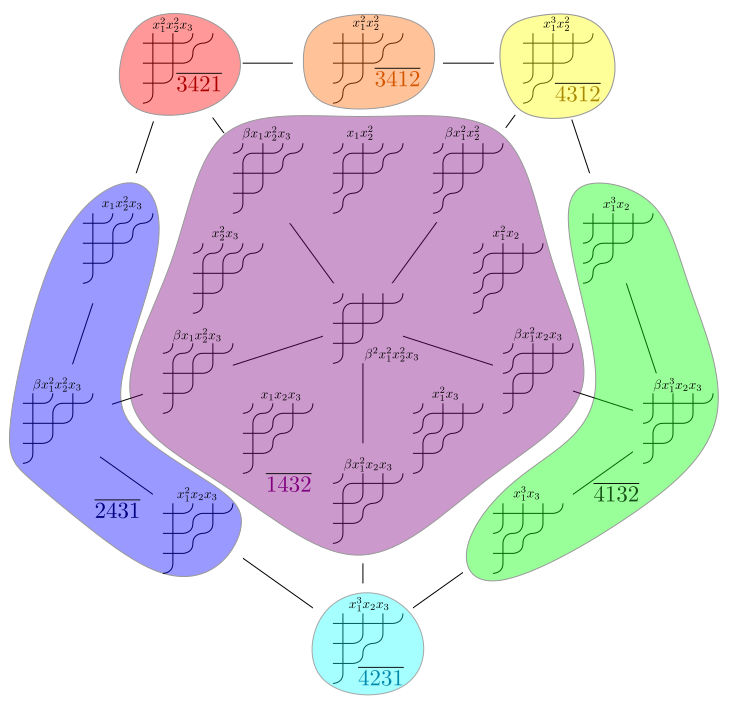}$$
		\caption{The pipe dream complex for $w=\overline{1432}$}
		\label{pd-complex}
	\end{figure}
\end{example}
For $k\in \mathbb Z_{\geq 0}$, we introduce the following notation: $\PD_k(w)=\{P\in \PD(w)\mid \mathrm{ex}(P)=k\}$.
\begin{corollary}\label{euler}
For each permutation $w\in\Sc_n$, we have
\[
\sum_{k=0}^{n(n-1)/2}(-1)^k |\PD_k(w)|=1.
\]
\end{corollary}
\begin{proof}
Specializing the Grothendieck polynomial at $\xx=(1,1,\ldots,1)$ and $\beta=-1$, we obtain the following relation:
\begin{multline*}
\Gf_w^{(-1)}(1,1,\ldots,1)=\sum_{P\in\PD(w)} (-1)^{\mathrm{ex}(P)}=\sum_{k=0}^{n(n-1)/2}(-1)^k |\PD_k(w)|=\sum_{P\in\mathrm{int}(\Delta(\cQ_{0,n}, w))} (-1)^{\mathrm{codim}(P)}=\\
=\sum_{P \in \Delta(\cQ_{0,n},w)}(-1)^{\mathrm{codim}(P)}+\sum_{P \in \partial\Delta(\cQ_{0,n},w)}(-1)^{\mathrm{codim}(P)}=(-1)^d\chi_{\Delta(\cQ_{0,n},w)}+(-1)^{d-1}\chi_{\partial\Delta(\cQ_{0,n},w)}.
\end{multline*}
Here $d=n(n-1)/2-\ell(w)$ is the dimension of the pipe dream complex, and  $\chi_\Delta$ stands for the Euler characteristic of $\Delta$.

For each $w\ne\delta(\cQ_{0,n})=w_0$, the  pipe dream complex $\Delta(\cQ_{0,n}, w)$ is homeomorphic to a $d$-dimensional ball, and its boundary is homeomorphic to a  $(d-1)$-dimensional sphere (for $w=w_0$, the corresponding pipe dream complex is a point). This means that 
\[
\chi_{\Delta(\cQ_{0,n},w)}=1,\]
\[\chi_{\partial\Delta(\cQ_{0,n},w)}=1-(-1)^d,
\]
and hence
\[
\Gf^{(-1)}(1,1,\ldots,1)=\sum_{k=0}^{n(n-1)/2}(-1)^k |\PD_k(w)|=(-1)^d\left(1-(1-(-1)^d)\right)=(-1)^{2d}=1.
\]
\end{proof}

\section{Slide complexes}\label{sec:sc}

In this section we provide the main construction of this paper: we define the stratification of subword complexes into strata corresponding to slide (or glide) orbits. These strata are called \emph{slide complexes}. We show that, just like subword complexes, all such strata are homeomorphic to balls or spheres.

\subsection{Slide complexes in general}\label{ssec:sc-gen}
As before, let $(\Pi, \Sigma)$ be a Coxeter system.
\begin{definition}
Let  $\cQ, \cS$ be two words in the alphabet $\Sigma$. By \emph{slide complex of subwords} $\widetilde{\Delta}(\cQ, \cS)$ we denote the set of subwords $\cQ\smallsetminus \cP$, such that their complements $\cP$ contain  $\cS$ as a subword. Similarly to the case of subword complexes, this set of subwords has a natural structure of a simplicial complex.
\end{definition}

The following theorem is similar to Theorem~\ref{thm:subword-shellable}. 
\begin{theorem}\label{thm:shellable-slides}
Slide complexes are vertex decomposable and hence shellable.
\end{theorem}

\begin{proof}
It is clear that slide complexes are pure.

Let $\cQ=(\sigma, \sigma_2, \ldots, \sigma_m)$ and $\cS=(s_{j_1},s_{j_2},\ldots,s_{j_l})$ be two words in the alphabet $\Sigma$. Let $\cQ'=(\sigma_2, \ldots, \sigma_m)$ and $\cS'=(s_{j_2},\ldots, s_{j_l})$. Then $\link(\sigma, \widetilde{\Delta}(\cQ, \cS))=\widetilde{\Delta}(\cQ',\cS)$. If the word $\cS$ starts with the letter $\sigma$, then $\del(\sigma, \widetilde{\Delta}(\cQ, \cS))= \widetilde{\Delta}(\cQ', \cS')$. Otherwise we have $\del(\sigma,  \widetilde{\Delta}(\cQ, \cS))=\link(\sigma, \widetilde{\Delta}(\cQ, \cS))= \widetilde{\Delta}(\cQ', \cS)$.
	
This means that for a vertex $\sigma$, the result of its deletion and its link in $\widetilde{\Delta}(\cQ, \cS)$ are slide complexes. Then we use the induction by the length of $\cQ$.
\end{proof}

Let us introduce an analogue of the Demazure product for subwords.
\begin{definition}
Denote by $\widetilde{\delta}(\cQ)$ the word  obtained from $\cQ$ by replacing each maximal subsequence of consecutive identical letters $s_i\ldots s_i$ by one letter $s_i$. 
	
For example,  $\widetilde\delta(s_1s_1s_2s_1s_2s_2s_2)=s_1s_2s_1s_2$.
\end{definition}

\begin{remark}
It is clear from the definition that for any word $\cQ$, one has $\widetilde\delta(\delta(\cQ))=\delta(\widetilde\delta(\cQ))=\delta(\cQ)$.
\end{remark}

The following result is the main theorem of this paper. It is analogous to~Thm~\ref{thm:ball-or-sphere} due to Knutson and Miller.

\begin{theorem}\label{thm:main}
Let $\cQ$ and $\cS$ be two words in the alphabet $\Sigma$, and let $\widetilde{\delta}(\cS)=\cS$. Then the slide complex $\widetilde{\Delta}(\cQ, \cS)$ is homeomorphic to a sphere if $\widetilde{\delta}(\cQ)=\cS$ and to a ball otherwise. A face $\cQ\smallsetminus \cP$ belongs to the boundary of this complex if and only if $\widetilde{\delta}(\cP)\ne \cS$.
\end{theorem} 

\begin{proof}  Consider the \emph{free Coxeter group} $\widehat \Pi$ generated by $\Sigma$ modulo the relations $s_i^2=e$, without any other relations (this corresponds to all marks at the edges of the Coxeter graph being equal to $\infty$). The elements of $\widehat \Pi$ can be naturally identified with words in the alphabet $\Sigma$ without identical consecutive letters. Under this identification, $\widetilde\delta(\cS)$ is just the Demazure product of the word $\cS$, and $\widetilde\delta(\cS)=\cS$ if and only if $\cS$ is reduced when considered as a word in $\widehat \Pi$.

This means that the slide complex $\widetilde{\Delta}(\cQ,\cS)$ is nothing but the subword complex $\Delta(\cQ,\cS)$ for the  group $\widehat\Pi$. The desired statement follows directly from~\cite[Thm~3.7, Cor.~3.8]{KnutsonMiller04}. 
\end{proof}


\begin{remark}
If $\widetilde{\delta}(\cS)\ne \cS$, a slide complex is not necessarily homeomorphic to a ball or to a sphere. For instance, if $\cQ=s_1s_1s_1s_1$ and $\cS=s_1s_1$, the complex $\widetilde{\Delta}(\cQ,\cS)$ is the 1-skeleton of a tetrahedron.
\end{remark}

\begin{remark} Another proof of Theorem~\ref{thm:main} can be obtained  by repeating the steps used in the proof of~\cite[Thm~3.7]{KnutsonMiller04}; essentially all the statements used in the latter theorem for the subword complexes also hold for the slide complexes. This proof is outlined in our short announcement~\cite{SmirnovTutubalina20}.
\end{remark}

Example~\ref{ex:pentagon} shows that the interior of the pipe dream complex for $w=\overline{1432}$ is decomposed into slide complexes; this can be viewed as a topological interpretation of the decomposition of the Schubert polynomial $\Sf_{1432}$ (resp. the Grothendieck polynomial $\Gf_{1432}^{(\beta)}$) into the sum of slide (resp. glide) polynomials. The following proposition  generalizes it for the case of arbitrary subword complexes; in the next subsection we will apply this proposition to the case of pipe dream complexes.

\begin{proposition}
	The interior part of the subword complex $\mathrm{int}\left(\Delta(\cQ,w)\right)$ can be decomposed into the disjoint union of the interior parts of slide complexes: 
	\begin{equation}
	\mathrm{int}\left(\Delta(\cQ,w)\right)=\bigsqcup_{\substack{S \textrm{ word in } \Sigma\\\widetilde{\delta}(\cS)=\cS\\\delta(S)=w}}\mathrm{int}\left(\widetilde\Delta(\cQ,\cS)\right).
	\label{eq:decomp}\end{equation}
\end{proposition}

\begin{proof}
Let $\cQ\smallsetminus \cP$ be an internal face of the subword complex of $w$. This means that $\delta(\cP)=w$.  Let $\cS=\widetilde{\delta}(\cP)$. It is clear that $\delta(\cS)=\delta(\cP)=w$, $\widetilde{\delta}(\cS)=\cS$, and hence $\cQ\smallsetminus \cP$ belongs to the interior of the slide complex $\widetilde\Delta(\cQ,\cS)$.

Let us show the converse. If  $\cQ\smallsetminus \cP$  is an interior face for the slide complex  $\widetilde\Delta(\cQ,\cS)$, with $\widetilde{\delta}(\cS)=\cS$  and $\delta(\cS)=w$, then $\widetilde\delta(\cP)=\cS$. This means that  $\delta(\cP)=\delta(\widetilde{\delta}(\cP))=\delta(\cS)=w$ and $\cQ\smallsetminus \cP$ is an interior face of the subword complex $\Delta(\cQ,w)$.

\end{proof}

\subsection{Slide complexes in pipe dream complexes}\label{ssec:scpd}
In this subsection we study the relation between the slide and glide orbits of pipe dreams and the slide complexes.

As we have seen, the pipe dreams of shape $w\in\Sc_n$, both reduced and non-reduced,  bijectively correspond to the internal faces of the pipe dream complex $\Delta(\cQ_{0,n}, w)$. This complex is homeomorphic to a ball unless $w=w_0$. 

\begin{proposition}
This decomposition of the interior part of $\Delta(\cQ_{0,n}, w)$ into the interior parts of slide complexes is consistent with the decomposition of the set $\PD(w)$ into glide orbits: the pipe dreams from each glide orbit bijectively correspond to the pipe dreams from the interior part of the corresponding slide complex.
\end{proposition}

\begin{proof}
 Let $P\in\PD(w)$ be a pipe dream corresponding to the subword $\mathrm{word}(P)$ in $\cQ_{0,n}=(\sigma_1,\sigma_2,\ldots,\sigma_{n(n-1)/2})$. 
 
Suppose that the action of the slide move  $S_i$ on $P$ is not identical: it moves a cross $(i,j)$ southwest to the position $(i+1,j-1)$.  Let $\sigma_k$ and $\sigma_{k+m}$ be two letters corresponding to the old and the new positions of this cross in $\cQ_{0,n}$ (both these letters are equal to $s_{i+j-1}$). Since the pipe dream $P$ has no crosses in the  $i$-th row to the left of the $j$-th column, and the row $i+1$ does not contain crosses to the right of the $(j-1)$-st column, this means that the letters $\sigma_{k+1},\ldots,\sigma_{k+m-1}$ do not occur in the subword $\word(P)$. The slide move  $S_i$ acts on  $\word(P)$ as follows: $\word(P)$ contains either both letters  $\sigma_k$ and $\sigma_{k+m}$, or only  $\sigma_k$. In the meantime,  $\word(S_i(P))$ contains only the letter $\sigma_{k+m}$, but not $\sigma_k$.
 
So each slide move either does not change the word $\word(P)$ or replaces two identical consecutive letters $s_{i+j-1}s_{i+j-1}$ in it by $s_{i+j-1}$. Thus slide moves preserve $\widetilde\delta$: we have $\widetilde{\delta}(\word(P))=\widetilde{\delta}(\mathrm{word}(S_i(P)))$ and for any two pipe dreams from the same glide orbit the corresponding faces belong to the interior of the same slide complex.

The converse is also true. Consider $P\in\PD(w)$. Suppose that $\sigma_k$ and $\sigma_{k+m}$ are two identical letters in $\cQ_{0,n}$, such that
\begin{itemize}
	\item the subword $\word(P)$ contains either $\sigma_k$ or both of them;
	\item the interval $\sigma_{k+1},\ldots,\sigma_{k+m-1}$ does not contain letters equal to $\sigma_k=\sigma_{k+m}$;
	\item the letters $\sigma_{k+1},\ldots,\sigma_{k+m-1}$ do not occur in the subword $\word(P)$.
\end{itemize}
Then replacing $\sigma_k$ or the pair of letters $(\sigma_k,\sigma_{k+m})$ in the subword $\word(P)$ by the letter $\sigma_{k+m}$ corresponds to a slide move applied to $P$ (and preserves $\widetilde{\delta}(\word(P))$). Now it is easy to note that a pipe dream $Q$ is quasi-Yamanouchi if and only if $\word(Q)=\cS$ does not contain consecutive identical letters and the subword $\word(Q)$ is the rightmost appearance (this means that no letter can be moved to the right by the described operation) of the word $\cS$ in $\cQ_{0,n}$. Of course, this rightmost appearance is unique.

So if $\widetilde{\delta}(\cS)=\cS$ and $\cQ_{0,n}$ contains $\cS$ as a subword, then there exists exactly one quasi-Yamanouchi pipe dream $Q$ such that $\word(Q)=\cS$. For a pipe dream $P\in\PD(w)$ such that $\widetilde{\delta}(\word(P))=\cS$ it can be easily seen that $\word(\dst(P))=\cS$ and thus $\dst(P)=Q$, so $P\in\dst^{-1}(Q)$.

Equivalently, if $\word(P_1)$ and $\word(P_2)$ belong to the interior of the same slide complex, then $\widetilde{\delta}(\word(P_1))=\widetilde{\delta}(\word(P_2))$ and $P_1$ and $P_2$ belong to the same glide orbit.
\end{proof}

This implies the following corollary, which is similar to Corollary~\ref{groth-sum}: it states that a glide polynomial is obtained as the sum of monomials over the interior faces of the corresponding slide complex.
\begin{corollary}\label{cor:glide}
Let $Q\in\QPD(w)$. Then
\[
\Gc_Q^{(\beta)}=\sum_{P\in\mathrm{int}(\widetilde\Delta(\cQ_{0,n}, \mathrm{word}(Q)))} \beta^{\mathrm{codim}(P)}\xx^P.
\]
Here by $\xx^P$ we denote the monomial for the pipe dream corresponding to the face $P$.
\end{corollary}

By specializing at $\beta=0$, we recover a similar statement for slide polynomials.

\begin{corollary}\label{cor:slide}
Let $Q\in\QPD_0(w)$. Then
\[
\Ff_Q=\sum_{P} \xx^P,
\]
where the sum is taken over all facets $P$ of the complex $\widetilde\Delta(\cQ_{0,n}, \mathrm{word}(Q))$.
\end{corollary}

	\begin{figure}[ht]
		$$\includegraphics[width=17cm]{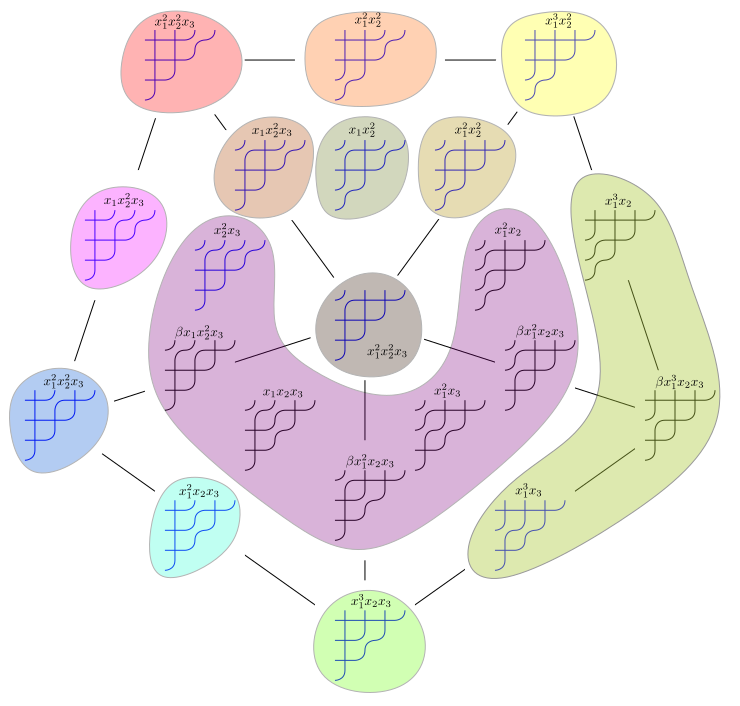}$$
		\caption{Pipe dream complex $w=\overline{1432}$ represented as a union of the interiors of slide complexes}
		\label{slide-complex}
	\end{figure}

\begin{example} Fig.\,\ref{slide-complex} represents the pipe dream complex for the permutation $w=\overline{1432}$ as a disjoint union of the interiors of slide complexes. The quasi-Yamanouchi pipe dreams are depicted in blue. For each pipe dream, we have the  monomial $\beta^{\mathrm{ex}(P)-\mathrm{ex}(Q)}\xx^P$ from the corresponding glide polynomial.
\end{example}

For $k\in\mathbb Z_{\geq 0}$, let $\QPD_k(w)=\{Q\in\QPD(w)\mid\mathrm{ex}(Q)=k\}$. The following corollary states that the alternating sum of the numbers of quasi-Yamanouchi pipe dreams with a given excess is 1.
\begin{corollary}
For each permutation $w\in \Sc_n$, we have the following equality:
\[
\sum_{k=0}^{n(n-1)/2}(-1)^k |\QPD_k(w)|=1.
\]
\end{corollary} 
\begin{proof}
Since the slide complexes are homeomorphic to balls, and the Euler characteristic of a ball is equal to 1, we obtain similarly to Corollary~\ref{euler} that
\[
\Gc^{(-1)}_Q(1,\ldots,1)=1
\]
	for each $Q\in \QPD(w)$. Let us specialize the equality
\[
	\Gf^{(\beta)}_w(\xx)=\sum_{Q\in\QPD(w)}\beta^{\mathrm{ex}(Q)}\Gc^{(\beta)}_Q(\xx)
\]
at $\beta=-1,\xx=(1,\ldots, 1)$, and use the fact that $\Gf^{(-1)}_w(1,\ldots,1)=\Gc^{(-1)}_Q(1,\ldots,1)=1$ for each $w\in\Sc_n, Q\in\QPD(w)$.  We obtain the desired formula:
\[
1=\sum_{Q\in\QPD(w)}(-1)^{\mathrm{ex}(Q)}=\sum_{k=0}^{n(n-1)/2}(-1)^k |\QPD_k(w)|.
\]
\end{proof}

\subsection{Remark on flip graphs}\label{ssec:pilaud}

V.\,Pilaud and C.\,Stump~\cite{PilaudStump13} describe an algorithm\footnote{We are grateful to the referee for bringing the paper~\cite{PilaudStump13} to our attention.} for indexing all the facets of a subword complex. This is done by constructing the so-called \emph{flip graph} of this complex. This graph, first defined in ~\cite[Rem.~4.5]{KnutsonMiller04}, is the facet adjacency graph of a subword complex: its vertices corresponding to facets of a subword complex, and two vertices are connected by an edge if the two corresponding facets share a codimension 1 face. This graph admits a canonical orientation, turning it into a poset. The arrows in this orientation are called \emph{increasing flips}.  This poset has a unique maximal and a unique minimal element, called the \emph{positive} (resp.~\emph{negative}) \emph{greedy facet}. The arrows of the opposite graph are called \emph{decreasing flips}; cf.~\cite[\S~4.2]{PilaudStump13}.

This construction is applicable to any Coxeter system, in particular, to the free Coxeter group $(\widehat Pi,\Sigma)$. It turns out that for this Coxeter system, one recovers the notion of slide moves. The following two propositions are immediate.

\begin{proposition} Let $Q$ be a reduced quasi-Yamanouchi pipe dream. Slide moves on the set $\dst_0^{-1}(Q)$ are precisely the decreasing flips of facets of  $\Delta(\cQ_{0,n},\mathrm{word}(Q))$ (or, equivalently, facets of $\int(\Delta(\cQ_{0,n},\mathrm{word}(Q)))$).
\end{proposition}

\begin{proposition} A reduced pipe dream $P$ is quasi-Yamanouchi if and only if the corresponding face of the slide complex $\Delta(\cQ_{0,n},\mathrm{word}(P))$ is its positive greedy facet.
\end{proposition}

\def\cprime{$'$}


\end{document}